\documentclass[11pt,a4paper,twoside]{article}
\usepackage[lmargin=25mm,rmargin=25mm,vmargin=25mm]{geometry}
\usepackage{indentfirst}
\usepackage[cp1250]{inputenc}     
\usepackage{amsfonts}
\usepackage{amssymb}
\usepackage{amsmath}
\usepackage{amsthm}

\newtheorem{tw}{Theorem}[section]
\newtheorem{prop}[tw]{Proposition}

\newtheorem{cor}[tw]{Corollary}
\theoremstyle{definition}
\newtheorem{deff}[tw]{Definition}
\newtheorem{ex}{Example}

\begin{document}

\title{Characterization of the Haagerup property \\ for residually amenable groups}
\author{Kamil Orzechowski\footnotemark}

\date{}

\maketitle

\footnotetext[1]{\footnotesize Faculty of Mathematics and Natural Sciences, University of Rzesz\'ow, Poland\\
Email address: kamil.orz@gmail.com}

\begin{abstract}
The notions of a box family and fibred cofinitely-coarse embedding are introduced. The countable, residually amenable groups satisfying the Haagerup property are then characterized as those possessing a box family that admits a fibred cofinitely-coarse embedding into a Hilbert space. This is a generalization of a result of X. Chen, Q. Wang and X. Wang \cite{Chen5} on residually finite groups.

\vspace{5mm}


\vspace{5mm}

\end{abstract}

\section{Introduction}

After M. Gromov published his monograph \textit{Asymptotic invariants of infinite groups}, the modern geometric group theory began to develop significantly. Since then, many large scale geometric properties of groups have been studied. Besides the classical notion of amenability (introduced by J. von Neumann), there are many large scale invariants which can be defined for groups (asymptotic dimension, hyperbolicity, Yu's property A, coarse embeddability into various spaces). For a systematic approach, we refer to \cite{Nowak}.

The Haagerup property (or Gromov's \textit{a-T-menability}) may be thought as  a weak form of amenability (all amenable groups have this property, as it was shown in \cite{Bekka}) or, in some sense, as an ``equivariant" version of coarse embeddability into Hilbert spaces. It was introduced (in different formulations) by U. Haagerup \cite{Haagerup} and M. Gromov \cite{Gromov}. This property has various applications in many fields of mathematics (harmonic analysis, ergodic theory, theory of operator algebras, the Baum-Connes conjecture).

Among discrete groups, residually finite groups are of a special interest. Recall that a countable group $G$ is residually finite if there exists a nested sequence of finite index normal subgroups such that the intersection of these subgroups is trivial. It is important that the notion of a box space can be introduced for such groups. The reason is that, geometrically, a box space reflects in some way the structure of a given residually finite group. For example, $G$ is amenable exactly when $\mbox{Box}(G)$ has Yu's property A \cite{Nowak}. X. Chen, Q. Wang and X. Wang in their paper \cite{Chen5} gave similar characterization for the Haagerup property, namely: $G$ has the Haagerup property if and only if $\mbox{Box}(G)$ admits a \textit{fibred coarse embedding} into a Hilbert space. The latter idea of a fibred coarse embedding was introduced by them in \cite{Chen4} to attack the maximal Baum-Connes conjecture.

We observe that the proof of Proposition 2.10 from \cite{Chen5} will be still valid if we replace standard averaging (division by the cardinality of a set) by averaging based on amenability. So, we could generalize the theorem to the class of residually amenable groups. Then, of course, we lose finiteness of quotients, so we need to permit constructing  box spaces from infinite components. To ensure that also the converse theorem holds, the fibred coarse embedding was modified to what we called a \textit{fibred cofinitely-coarse embedding}. This is a property of metric families (i.e. families of metric spaces) and it can be specified for the case of \textit{box families}, which are essentially analogs of box spaces. In effect, we got a sufficient and necessary condition for a countable, residually amenable group to have the Haagerup property. The result of \cite{Chen5} can be reconstructed as a special case of our slight generalization.

It is not clear how much our notion of fibred cofinitely-coarse embedding can be useful for other applications. We have proved that it is, with some restrictions, invariant under coarse equivalences between metric families. It may be interesting if another theorems involving box spaces can be generalized to the case of residually amenable groups.

The author would like to express his gratefulness to Professor Michael Zarichnyi from the Rzeszów University for his supervision, and to Professor Bronisław Wajnryb, Andrzej Wiśnicki, Paweł Witowicz and Janusz Dronka from the Rzeszów University of Technology, who had earlier made the author interested in the Haagerup property and related topics.

\section{Residually amenable groups}

\begin{deff}\label{rag}
We say that a countable group $G$ is \textit{residually amenable} if there exists a sequence $(G_n)_{n\in \mathbb{N}}$ of normal subgroups in $G$, such that $G_{n+1}\subset G_n$ for $n\in \mathbb{N}$, $\bigcap \limits_{n\in \mathbb{N}} G_n=\{1_G\}$ and each quotient group $G/G_n$ is amenable.
\end{deff}

As we are considering here only countable groups, the definition above is equivalent to a more general one. Namely, a group $G$ is residually amenable if and only if for any $g\neq 1_G$ there exists a normal subgroup $N$ such that $g\notin N$ and $G/N$ is amenable. Another equivalent formulation: for any $g\neq 1_G$ there exists a homomorphism $f\colon G \to H$ into an amenable group $H$ such that $f(g)\neq 1_H$.

\begin{ex}
Naturally, all residually finite groups are residually amenable. The most significant examples are free groups of countable rank and finitely generated linear groups. There is an important fact that the automorphism group of a finitely generated residually finite group is residually finite. Moreover, it can be proved that any finitely generated residually finite group must be Hopfian. For details, see \cite{MagKarSol} and \cite{RF}.
\end{ex}

\begin{ex}
Since solvable groups are amenable, every residually solvable group is residually amenable \cite{Resam}. Particularly, \textit{Baumslag-Solitar groups} $\mbox{BS}(m,n)$ belong to this class. It is shown that $\mbox{BS}(m,n)$ is residually finite if and only if $|m|=1$, $|n|=1$ or $|n|=|m|$. Thus, for other cases, we obtain examples of residually amenable groups which are not residually finite.
\end{ex}

\begin{ex}
There is a related class of groups, namely \textit{initially subamenable} or \textit{locally embeddable in amenable (LEA)} groups. A group $G$ is LEA if for every finite subset $F\subset G$ there exists an injection $f\colon F \to H$ into an amenable group $H$ such that $f(xy)=f(x)\cdot f(y)$ for $x,y,xy\in F$. In the class of finitely presented groups, the LEA property is equivalent to residual amenability. The so called \textit{sofic groups} are a further generalization. Examples of finitely presented sofic groups which are not LEA and, equivalently, not residually amenable have been constructed \cite{LEA}.
\end{ex}

Let $G$ be a residually amenable group with a fixed sequence $(G_n)_{n\in \mathbb{N}}$ of normal subgroups satisfying the requirements of Definition \ref{rag}. Equip $G$ with a proper length function $l$ and its associated metric $d(x,y):=l(x^{-1}y)$, ($x,y \in G$). Then every quotient group $G/G_n$ becomes a metric space with the metric $d_n$ induced by the length function $l_n$ defined by the formula:
$$l_n([g]):=\min \{l(gh) \colon h\in G_n\}$$
for all $[g]\in G/G_n$. In that manner, we have obtained a family of metric spaces. Let us introduce some terminology in the following definition.
\begin{deff}
The \textit{box family} of $G$ corresponding to $(G_n)_{n\in \mathbb{N}}$ is the metric family \\ $\left\{\left(G/G_n,d_n\right)\colon n\in \mathbb{N}\right\}$.
\end{deff}

We can also make a single metric space from all quotients $G/G_n$, rather than consider them as elements of a metric family.

\begin{deff}
The \textit{box space} of $G$ corresponding to $(G_n)_{n\in \mathbb{N}}$ is the set $\mbox{Box}(G):=\bigsqcup \limits_{n\in \mathbb{N}} G/G_n$ endowed with the metric $d'$, defined as follows:
\begin{itemize}
\item the restriction of $d'$ to each $G/G_n$ is the original metric $d_n$,
\item $d'(x,y):=l_n(x)+l_m(y)+n+m$ for $x\in G/G_n$, $y\in G/G_m$ and $n\neq m$.
\end{itemize}
\end{deff}

It can be shown that the function $d'$ is indeed a metric. We remark that the box space satisfies the condition
$$d'(G/G_n,G/G_m):=\inf \{d'(x,y)\colon x\in G/G_n,\, y\in G/G_m\} \rightarrow \infty$$
as $n+m \rightarrow \infty$ and $n\neq m$.

\section{Fibred coarse embedding and its generalization}

Now we are going to recall the notion of a \textit{fibred coarse embedding} \cite{Chen5,Chen4} and then modify it slightly so that it would be useful for metric families. 

\begin{deff}[\cite{Chen5}]\label{fce}
A metric space $(X,d)$ is said to admit a \textit{fibred coarse embedding} into a metric space $(Y,d_1)$ if there exist:
\begin{itemize}
\item a field of metric spaces $(Y_x)_{x\in X}$ such that each $Y_x$ is isometric to $Y$,
\item a \textit{section} $s\colon X \to \bigsqcup\limits_{x\in X} Y_x$ satisfying $s(x)\in Y_x$,
\item two non-decreasing functions $\rho_1,\rho_2\colon [0,\infty)\to [0,\infty)$ with $\lim \limits_{r\rightarrow \infty} \rho_1(r)=\infty$
\end{itemize}
such that, for any $r>0$, there exists a bounded subset $K_r\subset X$ for which, for any subset $C\subset X\setminus K_r$ of diameter less than $r$, there exists a \textit{trivialization}
$$t_C\colon \bigsqcup \limits_{x\in C} Y_x \to C\times Y,$$
such that the restriction of $t_C$ to any fibre $Y_x$, $x\in C$, is an isometry $t_C(x)\colon Y_x \to Y$ satisfying the following conditions:
\begin{enumerate}
\item for any $x,y \in C$: $\rho_1(d(x,y))\leq d_1(t_C(x)(s(x)),t_C(y)(s(y)))\leq \rho_2(d(x,y))$,
\item for any two subsets $C_1, C_2 \subset X\setminus K_r$ of diameter less than $r$ with $C_1 \cap C_2 \neq \emptyset$ there exists an isometry $t_{C_{1}C_2}\colon Y\to Y$ such that $t_{C_1}(x)\circ t^{-1}_{C_2} (x) = t_{C_{1}C_2}$ for all $x\in C_1 \cap C_2$.
\end{enumerate}
\end{deff}

\begin{deff}\label{fcf-ce}
A family $\mathcal{X}$ of disjoint metric spaces is said to admit a \textit{fibred cofinitely-coarse embedding} into a metric space $(Y,d_1)$ if there exist:
\begin{itemize}
\item a field of metric spaces $(Y_x)_{x\in \bigcup \mathcal{X}}$ such that each $Y_x$ is isometric to $Y$,
\item a \textit{section} $s\colon \bigcup \mathcal{X} \to \bigsqcup\limits_{x\in \bigcup \mathcal{X}} Y_x$ satisfying $s(x)\in Y_x$,
\item two non-decreasing functions $\rho_1,\rho_2\colon [0,\infty)\to [0,\infty)$ with $\lim \limits_{r\rightarrow \infty} \rho_1(r)=\infty$
\end{itemize}
such that, for any $r>0$, there exists a finite subfamily $\mathcal{K}_{r}\subset \mathcal{X}$ such that, for any metric space $(X,d)\in \mathcal{X}\setminus \mathcal{K}_{r}$ and for any subset $C\subset X$ of diameter less than $r$, there exists a \textit{trivialization}
$$t_C\colon \bigsqcup \limits_{x\in C} Y_x \to C\times Y,$$
such that the restriction of $t_C$ to any fibre $Y_x$, $x\in C$, is an isometry $t_C(x)\colon Y_x \to Y$ satisfying the following conditions:
\begin{enumerate}
\item for any $x,y \in C$: $\rho_1(d(x,y))\leq d_1(t_C(x)(s(x)),t_C(y)(s(y)))\leq \rho_2(d(x,y))$,
\item for any two subsets $C_1, C_2 \subset X$ of diameter less than $r$ with $C_1 \cap C_2 \neq \emptyset$ there exists an isometry $t_{C_{1}C_2}\colon Y\to Y$ such that $t_{C_1}(x)\circ t^{-1}_{C_2} (x) = t_{C_{1}C_2}$ for all $x\in C_1 \cap C_2$.
\end{enumerate}
\end{deff}

To motivate the idea of a fibred cofinitely-coarse embedding, we will show that it is invariant in some coarse sense. That means, this property is preserved under coarse equivalences of metric families, with some additional restrictions. Assume, we have two metric families $\tilde{\mathcal{X}}$ and $\mathcal{X}$. By a \textit{map of families} $\mathcal{F}\colon \tilde{\mathcal{X}}\to \mathcal{X}$ we understand a collection of functions such that each $f\in \mathcal{F}$ is a map $f\colon \tilde{X} \to X$ for some $\tilde{X}\in \tilde{\mathcal{X}}$, $X\in \mathcal{X}$, and each $\tilde{X} \in \tilde{\mathcal{X}}$ is the domain of at least one $f\in \mathcal{F}$.

\begin{deff}
A map of metric families $\mathcal{F}\colon \tilde{\mathcal{X}}\to \mathcal{X}$ is a \textit{coarse embedding} if there exist two non-decreasing functions $m,M\colon [0,\infty)\to [0,\infty)$ with $\lim \limits_{r\rightarrow \infty} m(r)=\infty$ such that for every $f\in \mathcal{F}$, where $f\colon \tilde{X} \to X$, and for any $\tilde{x},\tilde{x}'\in \tilde{X}$:
$$m\left(d_{\tilde{X}}(\tilde{x},\tilde{x}')\right)\leq d_{X}(f(\tilde{x}), f(\tilde{x}'))\leq M\left(d_{\tilde{X}}(\tilde{x},\tilde{x}')\right).$$
We call $\mathcal{F}$ a \textit{coarse equivalence} if additionally every $X\in \mathcal{X}$ is the codomain of some $f\in \mathcal{F}$ and there exists a constant $C>0$ such that, for any $f\in \tilde{X}\to X$ in $\mathcal{F}$, $f(\tilde{X})$ is a $C$-net in $X$.
\end{deff}

\begin{prop}
Let a map of metric families (each of them being disjoint if considered separately) $\mathcal{F}\colon \tilde{\mathcal{X}}\to \mathcal{X}$ be a coarse embedding such that each $X\in \mathcal{X}$ is the codomain of only finitely many functions $f\in \mathcal{F}$. If $\mathcal{X}$ admits a fibred cofinitely-coarse embedding into $(Y,d_1)$, then so does $\tilde{\mathcal{X}}$.
\end{prop}

\begin{proof}
Assume that $(Y_x)_{x\in \bigcup \mathcal{X}}$, a section $s$ and functions $\rho_1,\rho_2$ fulfill Definition \ref{fcf-ce}. For the metric family $\tilde{\mathcal{X}}$, we define $(Y_{\tilde{x}})_{\tilde{x}\in \bigcup \tilde{\mathcal{X}}}$ putting $Y_{\tilde{x}}:=Y_{f(\tilde{x})}$ where $f\in \mathcal{F}$ is a function $f\colon \tilde{X}\to X$ from the set $\tilde{X}\in \tilde{\mathcal{X}}$ containing $\tilde{x}$ (if there is some ambiguity, the choice of $f$ can be made arbitrarily). Then, we define the section $\tilde{s}(\tilde{x}):=s(f(\tilde{x}))$.
Fix $r>0$. Define $\tilde{\mathcal{K}}_r \subset \tilde{\mathcal{X}}$ to be the family of all $\tilde{X}\in \tilde{\mathcal{X}}$ which are mapped to some $X\in \mathcal{K}_{M(r)+1}$ by some $f\in \mathcal{F}$. By assumption, $\tilde{\mathcal{K}}_r$ is finite. Let $(\tilde{X},\tilde{d})\in \tilde{\mathcal{X}}\setminus \tilde{\mathcal{K}}_r$ and $\tilde{C}\subset \tilde{X}$ with diameter less than $r$. Then $\mbox{diam}(f(\tilde{C}))\leq M(r) < M(r)+1$ (for any $f\colon \tilde{X}\to X$) and $f(\tilde{C})$ does not lie in any set from $\mathcal{K}_{M(r)+1}$, so that we can ``lift" the corresponding trivializations, namely $t_{\tilde{C}}(\tilde{x}):=t_{f(\tilde{C})}(f(\tilde{x}))$. Then, for $\tilde{x},\tilde{y} \in \tilde{C}$, we have
$$\rho_1(d(f(\tilde{x}),f(\tilde{y})))\leq d_1(t_{\tilde{C}}(\tilde{x})(\tilde{s}(\tilde{x})),t_{\tilde{C}}(\tilde{y})(\tilde{s}(\tilde{y})))\leq \rho_2(d(f(\tilde{x}),f(\tilde{y}))).$$
By the definition of a coarse embedding, we obtain finally
$$(\rho_1\circ m) (\tilde{d}(\tilde{x},\tilde{y}))\leq d_1(t_{\tilde{C}}(\tilde{x})(\tilde{s}(\tilde{x})),t_{\tilde{C}}(\tilde{y})(\tilde{s}(\tilde{y})))\leq (\rho_2\circ M) (\tilde{d}(\tilde{x},\tilde{y})),$$
so $\rho_1 \circ m$ and $\rho_2 \circ M$ are such as required in Definition \ref{fcf-ce}. The remaining condition (the case when trivializations coincide) is obviously satisfied .
\end{proof}

In the following proposition, we are going to prove a connection between the two definitions stated above in the case of box spaces and box families.

\begin{prop}
If a box space $\mathrm{Box}\,(G)$ of a residually amenable group $G$ admits a fibred coarse embedding into $(Y,d_1)$, then the corresponding box family $\mathcal{X}:=\left\{\left(G/G_n,d_n\right)\colon n\in \mathbb{N}\right\}$ admits a fibred cofinitely-coarse embedding into $(Y,d_1)$.
\end{prop}

\begin{proof}
Since the union $\bigcup \mathcal{X}$ is (as a set) exactly the box space, we can choose the same field of metric spaces $(Y_x)$ and the same section $s$ which fulfilled Definition \ref{fce} for $\mbox{Box}(G)$. We also do not need to change the controlling functions $\rho_1,\rho_2$. Then, fix $r>0$. Since $K_r$ is a bounded set and $d'(G/G_n,G/G_m) \rightarrow \infty$, $K_r$ can intersect only finitely many components of the disjoint union $\bigsqcup \limits_{n\in \mathbb{N}} G/G_n$. We denote by $\mathcal{K}_r$ the finite family of these components. Thus, if $C\subset X\in \mathcal{X}\setminus \mathcal{K}_r$, then $C\subset \mbox{Box}(G)\setminus K_r$. Therefore, the trivializations $t_C$ can be  transferred to the case of the box family without any changes and the required conditions will be obviously satisfied.
\end{proof}

We notice that the reverse implication holds provided that the elements of the box family are bounded. It is the case when the quotients $G/G_n$ are finite, i.e. the group $G$ is residually finite. Thus, we have the following corollary:

\begin{cor}\label{finite}
For any residually finite group the concepts of a fibred coarse embedding (for a box space) and a fibred cofinitely-coarse embedding (for the corresponding box family) are equivalent.
\end{cor}

\section{Characterization of the Haagerup property}

The \textit{Haagerup property}, which is often called \textit{a-T-menability}, has many different formulations. For a comprehensive description of the property and its motivation, as well as fundamental results, we refer to \cite{Haag}. Here, we are going to state two equivalent defining conditions. It is remarkable that the Haagerup property may be defined for topological groups, but we restrict our attention to discrete groups.

\begin{deff}[\cite{Haag}]
A countable group $G$ is said to have the \textit{Haagerup property} if it satisfies one of the following equivalent conditions:
\begin{enumerate}
\item there exists a proper function $\psi \colon G \to [0,\infty)$ which is conditionally negative definite, that is, $\psi(g^{-1})=\psi(g)$ for all $g\in G$, and for all $g_1,\dots,g_n \in G$ and all $a_1,\dots, a_n \in \mathbb{R}$ with $\sum a_i =0$,
$$\sum_{i,j} a_i a_j \psi(g_{i}^{-1} g_j)\, \leq 0;$$
\item $G$ is \textit{a-T-menable}: there exists a Hilbert space $H$ and an isometric affine action $\alpha$ of $G$ on $H$ which is proper in the sense that, for all pairs of bounded subsets $B$ and $C$ of $H$, the set of elements $g\in G$ such that $(\alpha(g))(B)\cap C \neq \emptyset$ is finite.
\end{enumerate}
\end{deff}

Firstly, we are going to generalize Proposition 2.10 from \cite{Chen5}, which gives a sufficient condition for the Haagerup property in the case of a residually finite group.

\begin{tw}\label{=>}
Let $G$ be a countable, residually amenable group. If a box family of $G$ admits a fibred cofinitely-coarse embedding into a Hilbert space, then $G$ has the Haagerup property.
\end{tw}

\begin{proof}
Let $(G_n)_{n\in \mathbb{N}}$ be a sequence of normal subgroups with amenable quotients, satisfying the conditions of Definition \ref{rag}, such that the associated box family $\mathcal{X}:=\left\{\left(G/G_n,d_n\right)\colon n\in \mathbb{N}\right\}$ admits a fibred cofinitely-coarse embedding into a (real) Hilbert space $H$.

Fix an integer $r>0$. Let $\mathcal{K}_r$ be such as in Definition \ref{fcf-ce}. Since $\mathcal{K}_r$ is a finite family, there exists $n_r$ such that $\left(G/G_{n_r},d_{n_r}\right) \in \mathcal{X}\setminus \mathcal{K}_r$. Moreover, the ball $B_G (1_G,2r)$ is finite and, using the assumption $\bigcap \limits_{n\in \mathbb{N}} G_n=\{1_G\}$, we can choose $n_r$ large enough so that $B_G (1_G,2r) \cap G_{n_r} = \{1_G\}$. Then the quotient map $\pi_{n_r}\colon G \to G/G_{n_r}$ is an isometry, if restricted to any subset of diameter less than $r$. For any $[x],[y]\in G/G_{n_r}$ with $d_{n_r}([x],[y])<r$ choose a subset $C\subset G/G_{n_r}$ of diameter less than $r$ containing $[x],[y]$ and define:
$$k_{r}([x],[y]):=\|t_C([x])(s([x])) - t_C([y])(s([y]))\|^2.$$
We will show that this definition does not depend on the choice of the subset $C$. Suppose that $C_1,C_2$ are two subsets of $G/G_{n_r}$ of diameters less than $r$, both containing $[x]$ and $[y]$, and $t_{C_1},t_{C_2}$ are the corresponding trivializations. By the definition of a fibred cofinitely-coarse embedding we have
$$t_{C_1}([x])\circ t_{C_2}^{-1}([x])=t_{C_1 C_2}=t_{C_1}([y])\circ t_{C_2}^{-1}([y]),$$
where $t_{C_1 C_2}\colon H \to H$ is an (affine) isometry.
So 
$$
\begin{array}{l}
\|t_{C_2}([x])(s([x])) - t_{C_2}([y])(s([y]))\|^2=\|(t_{C_1 C_2} \circ t_{C_2}([x]))(s([x])) - (t_{C_1 C_2}\circ t_{C_2}([y]))(s([y]))\|^2\\
=\|t_{C_1}([x])(s([x])) - t_{C_1}([y])(s([y]))\|^2.
\end{array}$$ Put $k_r([x],[y])=0$ if $d_{n_r}([x],[y])\geq r$. Then $k_r$ is a well-defined function on $G/G_{n_r} \times G/G_{n_r}$. We will show that this function is an $r$-locally conditionally negative definite kernel on $G/G_{n_r} \times G/G_{n_r}$. Choose a finite sequence of points $[x_1],\dots , [x_m]$ in $G/G_{n_r}$ with $d_{n_r}([x_i],[x_j])<r$ (for $1\leq i,j \leq m$) and a sequence of real scalars $\lambda_1,\dots , \lambda_m$ with $\sum\limits_{i=1}^{m}\lambda_i=0$. Compute the expression
$$\begin{array}{l}
\sum\limits_{1\leq i,j \leq m} \lambda_i \lambda_j \, k_r([x_i],[x_j]) = \sum\limits_{1\leq i,j \leq m} \lambda_i \lambda_j \, \|t_C([x_i])(s([x_i]))\|^2 + \sum\limits_{1\leq i,j \leq m} \lambda_i \lambda_j \, \|t_C([x_j])(s([x_j]))\|^2 \\ 
-2\sum\limits_{1\leq i,j \leq m} \lambda_i \lambda_j \, \left\langle \, t_C([x_i])(s([x_i])), \, t_C([x_j])(s([x_j])\right\rangle \,.
\end{array}$$
The first two sums vanish due to the assumption $\sum\limits_{i=1}^{m}\lambda_i=0$ and we have
$$\begin{array}{l}
\sum\limits_{1\leq i,j \leq m} \lambda_i \lambda_j \, k_r([x_i],[x_j])=-2\sum\limits_{1\leq i,j \leq m} \lambda_i \lambda_j \, \left\langle \, t_C([x_i])(s([x_i])), \, t_C([x_j])(s([x_j])\right\rangle \, \\
= -2 \left\langle \, \sum\limits_{1\leq i \leq m} \lambda_i \, t_C([x_i])(s([x_i])), \, \sum\limits_{1\leq i \leq m} \lambda_i \, t_C([x_i])(s([x_i])) \right\rangle \, \leq 0.
\end{array}$$
Should $\rho_1,\rho_2$ be the controlling function from Definition \ref{fcf-ce}, we obtain
\begin{equation}\label{controlling}
(\rho_1 (d_{n_r}([x],[y])))^2 \leq k_r ([x],[y]) \leq (\rho_2 (d_{n_r}([x],[y])))^2
\end{equation}
for any $[x],[y] \in G/G_{n_r}$ with $d_{n_r}([x],[y])<r$.

Now we are going to define a function $\phi_r \colon B_{G/G_{n_r}}(1_{G/G_{n_r}},r) \to [0,\infty)$. We refer to amenability of $G/G_{n_r}$. By one of equivalent definitions of amenability, there exists a finitely additive measure $\mu_{n_r} \colon 2^{G/G_{n_r}} \to [0,1]$ which is right-invariant (i.e. $\mu_{n_r} (A[x]) = \mu_{n_r} (A)$ for $A\subset G/G_{n_r}$, $[x]\in G/G_{n_r}$) and normalized (i.e. $\mu_{n_r} (G/G_{n_r})=1$). This measure determines (via Lebesgue integration) a linear, order-preserving functional on $l^{\infty} (G/G_{n_r})$ which is invariant under composition with right multiplication in the group. For $[x]\in B_{G/G_{n_r}}(1_{G/G_{n_r}},r)$ define
$$\phi_r ([x]):=\int\limits_{G/G_{n_r}} k_r ([t],[tx]) \, d\mu_{n_r}([t]).$$
The integrated function is bounded because $d_{n_r} ([t], [tx]) = d_{n_r} (1_{G/G_{n_r}},[x])<r$ and we can apply \eqref{controlling} to obtain
$(\rho_1 (l_{n_r} ([x])))^2 \leq k_r ([t],[tx]) \leq (\rho_2 (l_{n_r}([x])))^2$ for all $[t]\in G/G_{n_r}$. Thus we have
\begin{equation}\label{c2}
(\rho_1 (l_{n_r} ([x])))^2 \leq \phi_r ([x]) \leq (\rho_2 (l_{n_r}([x])))^2.
\end{equation}
Using the right-invariance of the integral we can also conclude that for any $[x],[y]\in G/G_{n_r}$ with $d_{n_r}([x],[y])<r$:
$$\phi_r ([x^{-1}y])=\int\limits_{G/G_{n_r}} k_r ([t],[tx^{-1}y])\, d\mu_{n_r}([t])=\int\limits_{G/G_{n_r}} k_r ([tx],[ty])\, d\mu_{n_r}([t]).$$
Remembering that the projection $\pi_{n_r}\colon G \to G/G_{n_r}$ is $r$-isometric, let
$$\psi_r: = \phi_r \circ \pi_{n_r}$$
on the ball $B_G(1_G,r)$ and $\psi_r \equiv 0$ outside this ball. We claim that $\psi_r \colon G \to [0,\infty)$ is an $r$-locally conditionally definite function on $G$. Choose a finite sequence of points $x_1,\dots , x_m$ in $G$ with $d(x_i,x_j)<r$ (for $1\leq i,j \leq m$) and a sequence of real scalars $\lambda_1,\dots , \lambda_m$ with $\sum\limits_{i=1}^{m}\lambda_i=0$. Then also $d_{n_r}([x_i],[x_j])=d(x_i,x_j)<r$ and
$$\begin{array}{l}
\sum\limits_{1\leq i,j \leq m} \lambda_i \lambda_j \, \psi_r (x_i^{-1}x_j)=\sum\limits_{1\leq i,j \leq m} \lambda_i \lambda_j \, \phi_r ([x_i^{-1}x_j])=\sum\limits_{1\leq i,j \leq m} \lambda_i \lambda_j \int\limits_{G/G_{n_r}} k_r ([tx_i],[tx_j]) \, d \mu_{n_r}([t])\\
=\int\limits_{G/G_{n_r}} \sum\limits_{1\leq i,j \leq m} \lambda_i \lambda_j \, k_r([tx_i],[tx_j])\, d \mu_{n_r}([t]) \leq 0.
\end{array}$$
Similarly, from \eqref{c2} we get (for $x$ in $B_G(1_G,r)$):
$$(\rho_1(l(x))^2=(\rho_1 (l_{n_r} ([x])))^2 \leq \phi_r ([x]) = \psi_r (x) \leq (\rho_2  (l_{n_r}([x])))^2=(\rho_2 (l(x)))^2.$$

The upper bound above is valid obviously for all $x\in G$ and does not depend on $r$, hence $(\psi_r (x))_{r\in \mathbb{N}}$ takes values in $[0,(\rho_2 (l(x)))^2]$. Since $G$ is countable, $\prod\limits_{x\in G} [0,(\rho_2 (l(x)))^2]$ is homeomorphic to the Hilbert cube. Since the latter is sequentially compact, there exists a subsequence of the sequence $(\psi_r)_{r\in \mathbb{N}}$ that converges pointwise to a function $\psi\colon G \to [0,\infty)$. After passing to the limit as $r\rightarrow \infty$, we obtain clearly
$$(\rho_1(l(x))^2 \leq \psi(x) \leq (\rho_2 (l(x)))^2,$$
which means (since $\lim \limits_{r\rightarrow \infty} \rho_1(r)=\infty$) that $\psi$ is a proper function.

It remains to show that $\psi$ is conditionally negative definite on the whole $G$. Indeed, for a finite sequence of points $x_1,\dots , x_m$ in $G$ and a sequence of real scalars $\lambda_1,\dots , \lambda_m$ with $\sum\limits_{i=1}^{m}\lambda_i=0$, there exists $N$ such that all the distances $d(x_i,x_j)$ are less than $N$. Thus for $r\geq N$ we can apply the fact that $\psi_r$ is $r$-locally conditionally negative definite, to obtain $\sum\limits_{1\leq i,j \leq m} \lambda_i \lambda_j \, \psi_r (x_i^{-1}x_j) \leq 0$. Eventually, passing to the limit of the suitable subsequence, we get $\sum\limits_{1\leq i,j \leq m} \lambda_i \lambda_j \, \psi(x_i^{-1}x_j) \leq 0$,
which ends the proof that $G$ has the Haagerup property.
\end{proof}

Now, we are moving on to the converse statement. The proof essentially applies methods used to prove Theorem 2.3 in \cite{Chen5}. Those ideas are applied here to the concrete situation and our notion of a fibred cofinitely-coarse embedding is used.

\begin{tw}\label{<=}
Let $G$ be a countable, residually amenable group. If $G$ has the Haagerup property, then any box family $\mathcal{X}$ of $G$ admits a fibred cofinitely-coarse embedding into some Hilbert space $H$.
\end{tw}

\begin{proof}
Let $(G_n)_{n\in \mathbb{N}}$ be a nested sequence of normal subgroups of $G$ with trivial intersection (we will not use the assumption on the quotients to be amenable) and $\mathcal{X}:=\left\{\left(G/G_n,d_n\right)\colon n\in \mathbb{N}\right\}$ the corresponding box family. By assumption, there exists a Hilbert space $H$ and a proper action of $G$ on $H$ by affine isometries, i.e. a homomorphism $\alpha \colon G \to \mbox{Iso}(H)$. We can decompose $\alpha(g)$ into its translational and linear part, as follows:
$$\alpha(g)(x)=b(g)+L(g)(x),$$
where $g\in G,\, x\in H$.
Then, the relation $b(gh)=b(g)+L(g)(b(h))$ is satisfied for all $g,h\in G$.

For each $n\in \mathbb{N}$ there is a natural action of $G_n$ on $G \times H$, explicitly we have
$$g(x,y):=(gx,\alpha(g)(y))$$
for $g\in G_n, \, x\in G,\, y\in H$. The orbit of $(x,y)$ will be denoted by $[(x,y)]$. If $[a]=\pi_n(a)=G_n a$ is an element of $G/G_n$ ($\pi_n$ is a quotient map), the action can be restricted to $G_n a \times H$ and we define
$$H_{[a]}:=(G_n a \times H)/G_n,$$
which means that $H_{[a]}$ is the orbit space. We define a metric on it:
$$d_{H_{[a]}}([(x,y)],[(x',y')]):=\|y'-\alpha(x'x^{-1})(y)\|.$$
We show that $d_{H_{[a]}}$ is well-defined. Let $x,x'\in G_n a,\, y,y'\in H,\, g,h\in G$. Then
$$
\begin{array}{l}
d_{H_{[a]}}([(gx,\alpha(g)(y))],[(hx',\alpha(h)(y'))])= \|\alpha(h)(y')-\alpha(hx'x^{-1}g^{-1})(\alpha(g)(y))\| \\
=\|y'-\alpha(x'x^{-1})(y)\|=d_{H_{[a]}}([(x,y)],[(x',y')])
\end{array}$$
because $\alpha(h)$ is an isometry. We shall only verify the triangle inequality:
$$
\begin{array}{l}
d_{H_{[a]}}([(x,y)],[(x'',y'')])=\|y''-\alpha(x''x^{-1})(y)\|\leq \|y''-\alpha(x''x'^{-1})(y')\| \\
+\|\alpha(x''x'^{-1})(y')-\alpha(x''x^{-1})(y)\|=d_{H_{[a]}}([(x',y')],[(x'',y'')])+\|\alpha(x'^{-1})(y')-\alpha(x^{-1})(y)\| \\
=d_{H_{[a]}}([(x',y')],[(x'',y'')])+\|\alpha(x'x'^{-1})(y')-\alpha(x'x^{-1})(y)\| \\
=d_{H_{[a]}}([(x',y')],[(x'',y'')])+d_{H_{[a]}}([(x,y)],[(x',y')]).
\end{array}$$
It is worth noting that, for any fixed $a_0\in G_n a$, the map
\begin{equation}\label{lift}
[(x,y)]\mapsto \alpha(a_0 x^{-1})(y)
\end{equation}
is an isometry from $H_{[a]}$ onto $H$.
Clearly, $\|\alpha(a_0 x'^{-1})(y') - \alpha(a_0 x^{-1})(y)\|=\|y' - \alpha(x'x^{-1})(y)\|=d_{H_{[a]}}([(x,y)],[(x',y')])$, and each $y\in H$ is the image of $[(a_0,y)]$.

So, we have obtained the field $(H_{[a]})_{[a]\in G/G_n}$ of metric spaces isometric to $H$. Letting $n$ vary, we get $(H_x)_{x\in \bigcup \mathcal{X}}$.
For $[a]\in G/G_n$, put $s([a]):=[(a,b(a))]\in H_{[a]}$. Since $[(ga,b(ga))]=[(ga,\alpha(g)(b(a)))]=[(a,b(a))]$, this is a well-defined section on $\bigcup \mathcal{X}$.

Let $r>0$. Take $n_r\in \mathbb{N}$ to be the smallest number such that, for $n\geq n_r$, $B_G (1_G,3r)\, \cap \, G_{n_r} = \{1_G\}$. Let now $\mathcal{K}_r:=\{G/G_n \colon n<n_r\}$. Choose a space $(G/G_n,d_n)\in \mathcal{X}\setminus \mathcal{K}_r$ and a subset $C\subset G/G_n$ with $\mbox{diam}(C)<r$. We are going to define a trivialization
$$t_C\colon \bigsqcup \limits_{x\in C} H_x \to C\times H.$$
Let $z\in G$ be an arbitrary point (``basepoint") of $\pi_n^{-1}(C)$. Then each $[a]\in C$ has a unique lift $a_0\in G$ such that $[a_0]=[a]$ and $d(z,a_0)<r$. Moreover, the lifting map $[a]\mapsto a_0$ is isometric on $C$. We define $t_C ([a])\colon H_{[a]} \to H$ to be the map given in \eqref{lift} with respect to our choice of $a_0$.
Now, for any $[a],[a']\in C$, we have
$$
\begin{array}{l}
\|t_C([a])(s([a]))-t_C([a'])(s([a']))\|=\|t_C([a])([(a,b(a))])-t_C([a'])([(a',b(a'))])\| \\
=\|\alpha(a_0 a^{-1})(b(a))-\alpha(a'_0 a'^{-1})(b(a'))\|=\|b(a_0)-b(a'_0)\| \\
=\|b(1_G)-b(a_0^{-1}a'_0)\|=\|b(a_0^{-1}a'_0)\|.
\end{array}$$
If we define $\rho_1(x):=\min\{\|b(g)\|\colon g\in G, l(g)\geq x\}$ and $\rho_2(x):=\max\{\|b(g)\|\colon g\in G, l(g)\leq x\}$ for all $x\in [0,\infty)$, then
$$
\begin{array}{l}
\|t_C([a])(s([a]))-t_C([a'])(s([a']))\|=\|b(a_0^{-1}a'_0)\|\in [\rho_1(d(a_0,a'_0)), \rho_2(d(a_0,a'_0))] \\
=[\rho_1(d_n([a],[a'])),\rho_2(d_n([a],[a']))].
\end{array}$$
Note that the controlling functions are independent on $n\in \mathbb{N}$. The condition $\lim \limits_{r\rightarrow \infty} \rho_1(r) =\infty$ is satisfied due to the fact that the action of $G$ on $H$ is proper. In the trivial situation when $G$ is finite, $\rho_1$ needs a slight modification to avoid taking the value $\infty$, because the set $\{\|b(g)\|\colon g\in G, \, l(g)\geq x\}$ is empty for large $x$. However, this is not a significant problem.

It only suffices to show that, whenever $C_1,C_2\subset G/G_n$ have diameters less than $r$ and $[a]\in C_1 \cap C_2$, the corresponding trivializations $t_{C_1}([a]), t_{C_2}([a])$ differ only by an isometric map $t_{C_1}([a]) \circ t^{-1}_{C_2}([a]) \colon H \to H$.
Let $z_1 \in \pi_n^{-1}(C_1)$ and $z_2 \in \pi_n^{-1}(C_2)$ be ``basepoints" and $a_1, a_2$ corresponding lifts of $[a]$. Then, for any $y\in H$, we have
$$(t_{C_1}([a]) \circ t^{-1}_{C_2}([a]) )(y)=t_{C_1}([a]) ( [(a_2,y)])=\alpha(a_1 a_2^{-1})(y).$$
By assumption, the map $\alpha(a_1 a_2^{-1}) \colon H \to H$ is an isometry, which completes the proof.
\end{proof}

Combining Theorems \ref{=>} and \ref{<=}, we can formulate a sufficient and necessary condition for a residually amenable group to have the Haagerup property.

\begin{cor}
A countable, residually amenable group has the Haagerup property if and only if one (or, equivalently, all) of its box families admits a fibred cofinitely-coarse embedding into a Hilbert space.
\end{cor}

Applying our Corollary \ref{finite}, we repeat the main statement from \cite{Chen5}:

\begin{cor}
A countable, residually finite group has the Haagerup property if and only if one (or, equivalently, all) of its box spaces admits a fibred coarse embedding into a Hilbert space.
\end{cor}

\end{document}